\documentclass{article}
\usepackage{amsfonts}
\usepackage{color}

\newtheorem{theorem}{Theorem}

\newtheorem{lemma}[theorem]{Lemma}

\newenvironment{proof}[1][Proof]{\noindent \textbf{#1.} }{\  \rule{0.5em}{0.5em}}

\begin{document}

\title{Continuity of reachable sets of restricted affine control systems}
\author{V\'{\i}ctor Ayala\thanks{%
Supported by Proyecto Fondecyt $n^{o}$ 1150292, Conicyt, Chile} \\
Universidad de Tarapac\'a\\
Instituto de Alta Investigaci\'on\\
Casilla 7D, Arica, Chile\\
and\\
Adriano Da Silva\\
Instituto de Matem\'atica,\\
Universidade Estadual de Campinas\\
Cx. Postal 6065, 13.081-970 Campinas-SP, Brasil.\\
}
\date{\today }
\maketitle

\begin{abstract}
In this paper we give a direct proof that for a restricted affine control
system on a connected manilfold $M$, the associated reachable sets up to
time $t$ varies contnuously with the Haudorff metric.
\end{abstract}

\textbf{Key words }Affine system, accesible sets, continuity and Hausdorff
metric

\bigskip
\textbf{2010 Mathematics Subject Classification }93B03, 93B99, 93C15

\section{Introduction}

Accesible sets of control system have been studied for many people. Just to
mention a few, the description of these class of sets have been analyzed by,
Darken \cite{Da}, Gronski \cite{Gr}, Lobry \cite{Lo} and Sussmann and
Jurdjevic \cite{Su}.

On the other hand, in his book \cite{Po} Pontryagin shows that for a
restricted linear systems on Euclidean spaces, the accessible sets deform
continuously with the Hausdorff metric. By a direct proof in this paper we
obtain the same results for a more general class of control systems.

Consider any restricted affine control system on a connected Riemannian $%
\mathcal{C}^{\infty }$-manifold $M,$ determined by the family of
differential equations 
\[
\Sigma _{\Omega }:\hspace{1cm}\dot{x}(t)=f_{0}(x(t))+%
\sum_{i=1}^{m}u_{i}(t)f_{i}(x(t)),\; \; \mbox{ where }\; \;u\in \mathcal{U}%
_{\Omega }.
\]%
Where $\mathcal{U}_{\Omega }:=\left \{ u\in L^{\infty }(\mathbb{R},\mathbb{R}%
^{m});\; \;u(t)\in \Omega \right \} $ is the class of admissible control
functions with $\Omega $ being a compact and convex subset of $\mathbb{R}%
^{m} $.

If $x\in M$ and $u\in \mathcal{U}_{\Omega},$ $\varphi (t,x,u)$ denotes the $\Sigma_{\Omega }$-solution satisfying $\varphi (0,x,u)=x.$ The reachable
set $\mathcal{R}_{\leq t,\Omega }(x)$\ of $\Sigma _{\Omega }$ is builded
with the points of $M$ which are possible to reach starting from the initial
condition $x$, through concatenation of $\Sigma _{\Omega }$-solutions in
nonnegative time less or equal than $t.$

It is well known that the map 
\[
(t,x,u)\in \mathbb{R}\times M\times \mathcal{U}_{\Omega }\mapsto \varphi
(t,x,u)\in M.
\]%
is continuous. Furthermore, the set $\mathcal{U}_{\Omega }$ is a compact
metrizable space in the weak* topology of $L^{\infty }(\mathbb{R},\mathbb{R%
}^{m})=L^{1}(\mathbb{R},\mathbb{R}^{m})^{\ast }$ (see for example \cite{Ka}).

In this paper we give a direct proof that for a restricted affine control
system $\Sigma _{\Omega }$\ on a connected manilfold $M$, the associated
reachable sets up to time $t$ varies continuously. Precisely, the map 
\[
(t,x, \Omega)\mapsto \mathcal{R}_{\leq t,\Omega }(x)\subset M
\]%
is continuous. 

In this case, the last variable belongs to the metric space $\left(\mathrm{Co}(\mathbb{R}^m), d_H\right)$ where

\[
\mathrm{Co}(\mathbb{R}^m):=\left \{ \Omega\subset
\mathbb{R}^{m}; \;\Omega\mbox{ is a non-empty compact convex subset}\right \}
\]
and $d_H$ is the Hausdorff metric. Moreover, $\mathcal{R}_{\leq t, \Omega}(x)\in \mathcal{C}(M)$ where $\left(\mathcal{C}(M), \varrho_H\right)$ is the metric space of all non-empty compact subsets of $M$ with the Hausdorff metric.  

As a consequence, any continuous functional $J$ defined on the accessible set $\mathcal{R}_{\leq t,\Omega }(x)$ has a minimum and maximum. In fact, $J\left(\mathcal{R}_{\leq t,\Omega }(x)\right)$ is
compact.

\section{Control affine systems}

Let $M$ be a connected Riemannian $\mathcal{C}^{\infty}$-manifold and $f_0,
f_1, \ldots, f_m\in \mathcal{X}^{\infty}(M)$ vector fields. A \textbf{%
control affine system} is the family of ordinary differential equations 
\[
\Sigma_{\Omega}:\hspace{1cm}\dot{x}(t)=f_0(x(t))+%
\sum_{i=1}^{m}u_i(t)f_i(x(t)), \; \; \mbox{ where }\; \;u\in \mathcal{U}%
_{\Omega}.
\]
The \textbf{set of the control functions} $\mathcal{U}_{\Omega}$ is defined
as 
\[
\mathcal{U}_{\Omega}:=\left \{u\in L^{\infty}(\mathbb{R}, \mathbb{R}^m); \;
\;u(t)\in \Omega \right \}
\]
with $\Omega$ being a compact and convex subset of $\mathbb{R}^m$. It is
well known that the set of the control functions is a compact metrizable
space in the weak* topology of $L^{\infty}(\mathbb{R}, \mathbb{R}^m)=L^1(%
\mathbb{R}, \mathbb{R}^m)^*$ (see for instance Proposition 1.14 of \cite{Ka}%
). For a given initial state $x\in M$ and $u\in \mathcal{U}_{\Omega}$ we
denote the solution of $\Sigma_{\Omega}$ by $\varphi(t, x, u)$. The curve $%
t\mapsto \varphi(t, x, u)$ is the only solution of $\Sigma_{\Omega}$
satisfying $\varphi(0, x, u)=x$ in the sense of Caratheod\'ory, that is, it
is an absolutely continuous curve that satisfies the corresponding integral
equation. Throughout we assume that all the solutions are defined in the
whole real line. Even though this assumption is in general restrictive, there are several cases where the assumption of completeness goes without loss of generality, such as linear systems on Lie groups and control affine systems on compact manifolds.

Moreover, the map 
\[
(t, x, u)\in \mathbb{R}\times M\times \mathcal{U}_{\Omega}\mapsto \varphi(t,
x, u)\in M
\]
is a continuous map (see for instance Theorem 1.1 of \cite{Ka}).

For a given state $x\in M$ we introduce the sets, 
\[
\mathcal{R}_{\leq t, \Omega}(x):=\left \{y\in M; \; \exists u\in \mathcal{U}%
_{\Omega}, s\in[0, t]\mbox{ with }\;y=\varphi(s, x, u)\right \}, \; \;t>0,
\]
\[
\mathcal{R}_{\Omega}(x):=\bigcup_{t>0}\mathcal{R}_{\leq t, \Omega}(x).
\]
The set $\mathcal{R}_{\leq t, \Omega}(x)$ is called the \textbf{set of
points reachable from $x$ up to time $t$}; the set $\mathcal{R}_{\Omega}(x)$
is called \textbf{the set of points reachable from $x$}.

By considering the set of the piecewise control functions $\mathcal{U}%
_{\Omega}^{\mathrm{PC}}\subset \mathcal{U}_{\Omega}$, the set 
\[
\mathcal{R}^{\mathrm{PC}}_{\leq t, \Omega}(x):=\left \{y\in M; \; \exists
u\in \mathcal{U}_{\Omega}^{\mathrm{PC}}, s\in[0, t]\mbox{ with }%
\;y=\varphi(s, x, u)\right \}
\]
satisfies, by Proposition 1.16 of \cite{Ka}, the following 
\[
\mathrm{cl}\left(\mathcal{R}_{\leq t, \Omega}^{\mathrm{PC}}(x)\right)=\mathcal{%
R}_{\leq t, \Omega}(x)\; \; \mbox{ for any }\; \;t>0, x\in M.
\]

Our aim in this paper is to show that the set of reachable points up to time 
$t$ varies continuously in the Hausdorff measure. In order to do that we
need to make some remarks. If $\Omega _{1}\subset \Omega _{2}\subset \mathbb{%
R}^{m}$ are compact convex subsets, we can consider the control affine
systems $\Sigma _{\Omega _{1}}$ and $\Sigma _{\Omega _{2}}$ with set of
control functions $\mathcal{U}_{\Omega _{1}}$ and $\mathcal{U}_{\Omega _{2}}$, respectively%
. Since by the very definition of such sets we have that $\mathcal{U}%
_{\Omega _{1}}\subset \mathcal{U}_{\Omega _{2}}$, the unicity of the
solutions imply that all the solutions of $\Sigma _{\Omega _{1}}$ are also
solutions of $\Sigma _{\Omega _{2}}$. Also, the set $\mathcal{U}_{\Omega
_{1}}$ is a compact subset of $\mathcal{U}_{\Omega _{2}}$ in the
weak*-topology.

Let us also remark that the sets 
\[
W_{u, \gamma}(x_1, \ldots, x_k):=\left \{u^{\prime }\in \mathcal{U}%
_{\Omega}; \, \left|\int_{\mathbb{R}}\langle u(s)-u^{\prime }(s),
x_i(s)\rangle \mathrm{ds}\right|<\gamma \mbox{ for } i=1, \ldots, k\right \}
\]
where $k\in \mathbb{N}$ and $x_i\in L^1(\mathbb{R}, \mathbb{R}^m)$ for $%
1\leq i\leq k$, form a subbasis for the weak*-topology (see \cite{Du}).

\section{Continuity}

\begin{theorem}
The map 
\[
(t, x, \Omega)\mapsto \mathcal{R}_{\leq t, \Omega}(x)
\]
is continuous in the Hausdorff measure.
\end{theorem}

The theorem follows from the next three lemmas:

\begin{lemma}
The map $\Omega \rightarrow \mathcal{R}_{t, \Omega}(x)$ is continuous.
\end{lemma}

\begin{proof}
Fix $\Omega$ and consider $\widehat{\Omega}$ such that $\Omega \subset \mathrm{int}\,\widehat{\Omega}$. Since 
\[
(t, x, u)\in \mathbb{R}\times M\times \mathcal{U}_{\widehat{\Omega}}\mapsto
\varphi(t, x, u)\in M
\]
is a continuous map we have, by fixing $x\in M$, that the map 
\[
(s, u)\in[0, t]\times \mathcal{U}_{\widehat{\Omega}}\mapsto \varphi_x(t,
u):=\varphi(t, x, u)\in M
\]
is uniformly continuous.

By Proposition 1.6 of \cite{Ka}, for any $u\in \mathcal{U}_{\Omega}$ there
exists a piecewise constant function $u^{\prime }\in \mathcal{U}_{\Omega}^{%
\mathrm{PC}}$ such that 
\[
\varrho(\varphi_x(s, u), \varphi_x(s, u^{\prime }))<\varepsilon/2, \; \;%
\mbox{ for any }\; \;s\in [0, t].
\]
Being that $\mathcal{U}_{\Omega}$ is a compact subset of $\mathcal{%
U}_{\widehat{\Omega}}$ we have by continuity that there are piecewise
constant function $u_1, \ldots, u_m\in \mathcal{U}_{\Omega}$ and $\gamma_1,
\ldots, \gamma_m>0$ such that 
\[
\mathcal{U}_{\Omega}\subset \bigcup_{i=1}^m W_{u_i, \gamma_i}(x_{i, 1},
\ldots, x_{i, k_i}), \;\;\;\mbox{ for }x_{i, j}\in L^1(\mathbb{R}, \mathbb{R}^m),
\]
and for any $s\in[0, t]%
, u\in W_{u_i, \gamma_i}(x_{i, 1}, \ldots, x_{i, k_i})$  we have that
\[
\varrho(\varphi_x(s, u), \varphi_x(s, u_i))<\varepsilon/2, \; \;1\leq i\leq
m.
\]

\textbf{Claim 1:} There exists $\epsilon>0$ such that for any $u\in \mathcal{%
U}_{\Omega}$ there exists $i^*\in \{1, \ldots, m\}$ such that 
\[
W_{u, \epsilon}\left(\Delta \right)\subset W_{u_{i^*}, \gamma_{i^*}}(x_{i^*,
1}, \ldots, x_{i^*, k_{i^*}})
\]
where $\Delta:=\{x_{i, j}, 1\leq j\leq k_i, 1\leq i\leq m\}$.

For any $u\in \mathcal{U}_{\Omega}$ let $\delta_u>0$ such that $W_{v,
\epsilon_u}(\Delta)\subset W_{u_i, \gamma_i}(x_{i, 1}, \ldots, x_{i, k_i})$
for some $1\leq i\leq m$. Since $\mathcal{U}_{\Omega}$ is compact, there
exist $v_1, \ldots, v_l\in \mathcal{U}_{\Omega}$ and $\epsilon_1,
\ldots, \epsilon_l>0$ such that 
\[
\mathcal{U}_{\Omega}\subset \bigcup_{k=1}^l W_{v_k, \epsilon_k/2}(\Delta) \;
\; \; \mbox{ with }\; \; \; W_{v_k, \epsilon_k}(\Delta)\subset W_{u_i,
\gamma_i}(x_{i, 1}, \ldots, x_{i, k_i}),
\]
where $1\leq i\leq m$. Let then $0<\epsilon<\epsilon_k/2$ for all $k\in \{1,
\ldots, l\}$.

For any $u\in \mathcal{U}_{\Omega}$, let $\bar{u}\in W_{u, \epsilon}(\Delta)$%
. Consider $k\in \{1, \ldots, l\}$ such that $u\in W_{v_k,
\epsilon_k/2}(\Delta)$. Then 
\[
\left|\int_{\mathbb{R}}\langle \bar{u}(s)-v_k(s), x_{i, j}(s)\rangle \mathrm{%
ds}\right|\leq
\]
\[
\left|\int_{\mathbb{R}}\langle \bar{u}(s)-u(s), x_{i, j}(s)\rangle \mathrm{ds%
}\right|+\left|\int_{\mathbb{R}}\langle u(s)-v_k(s), x_{i, j}(s)\rangle%
\mathrm{ds}\right|<\epsilon+\epsilon_k/2<\epsilon_k,
\]
showing that $W_{u, \epsilon}(\Delta)\subset W_{v_k,
\epsilon_k}(\Delta)\subset W_{u_{i^*}, \gamma_i}(x_{i, 1}, \ldots, x_{i^*, k_{i^*}})$
for some $i^*\in \{1, \ldots, m\}$ as stated.

\bigskip

\textbf{Claim 2:} For any $\varepsilon>0$ there exists $\delta_1>0$ such
that if $\Omega^{\prime }$ and $\gamma \in(0, \delta_1)$ are such that $%
N_{\gamma}(\Omega^{\prime })\subset \mathrm{int}\,\widehat{\Omega}$ then 
\[
\Omega \subset N_{\gamma}(\Omega^{\prime })\Rightarrow \mathcal{R}_{\leq t,
\Omega}(x)\subset N_{\varepsilon}(\mathcal{R}_{\leq t, \Omega^{\prime }}(x)).
\]

Let $T>t$ be such that 
\[
\int_{\mathbb{R}\setminus[-T, T]}\left|x_{i, j}(s)\right|\mathrm{ds}<\frac{%
\epsilon}{2\mathrm{diam}\widehat{\Omega}}, \; \mbox{ for }1\leq j\leq k_i,
\;1\leq i\leq m.
\]
and denote by 
\[
M:=\max \left \{ \int_{-T}^T\left|x_{i, j}(s)\right|\mathrm{ds}, \; \;1\leq
i\leq k_i, 1\leq j\leq m \right \}.
\]
Consider $\delta_1>0$ be such that $\delta_1<\epsilon/2M$ for $1\leq i\leq m$
and let $\Omega^{\prime }$ and $\gamma \in (0, \delta_1)$ satisfying $\Omega
\subset N_{\gamma}(\Omega^{\prime })\subset \mathrm{int}\widehat{\Omega}$.

Since $u_i$ is a piecewise continuous, it assume finite values $c_{1, i}, \ldots,
c_{n_i, i}\in \Omega$ when restricted to $[-T, T]$. Let $c^{\prime }_{1, i}, \ldots, c^{\prime }_{n_i, i}\in \Omega^{\prime }$ be
such that $|c_{j, i}- c_{j, i}^{\prime }|<\gamma$ for $j=1, \ldots, n_i$.
Define 
\[
u^{\prime }_i(s):=\left \{%
\begin{array}{cc}
c^{\prime }_{i, j}, & \mbox{ if } u_i(s)=c_{i, j}, \; \; \mbox{ and }\; \; s%
\in[0, t] \\ 
p^{\prime }_i, & \mbox{ for } s\in \mathbb{R}\setminus [-T, T],%
\end{array}%
\right.
\]
where $p^{\prime }_i\in \Omega^{\prime }$ are arbitrary points. The
functions $u^{\prime }_i$ for $i=1, \ldots, m$ are piecewise constant
functions which implies that $u^{\prime }_i\in \mathcal{U}_{\Omega^{\prime
}} $. Moreover, for $1\leq j\leq k_i$ 
\[
\left|\int_{\mathbb{R}}\langle u_i(s)-u^{\prime }_i(s), x_{i, j}(s)\rangle 
\mathrm{ds}\right|
\]
\[
<\gamma \int_{-T}^T\left|x_{i, j}(s)\right| \mathrm{ds}+\mathrm{diam}\widehat{%
\Omega}\int_{\mathbb{R}\setminus[-T, T]}\left| x_{i, j}(s)\right|\mathrm{ds}%
<\epsilon<\gamma_i
\]

showing that $u^{\prime }_i\in W_{u_i, \gamma_i}(x_{i, 1}, \ldots, x_{i,
k_i})$.

\bigskip

Consider now $y\in \mathcal{R}_{\leq t, \Omega}(x)$ and let $u\in \mathcal{U}%
_{\Omega} $, $s\in [0, t]$ and $i\in \{1, \ldots, m\}$ such that 
\[
y=\varphi_{x}(s, u) \; \; \; \mbox{ with }\; \; \;u\in W_{u_i,
\gamma_i}(x_{i, 1}, \ldots, x_{i, k_i}).
\]
We have that 
\[
\varrho(\varphi_x(s, u^{\prime }_i), \varphi_x(s, u))\leq
\varrho(\varphi_x(s, u^{\prime }_i), \varphi_x(s, u_i))+\varrho(\varphi_x(s,
u_i), \varphi_x(s, u))<\varepsilon.
\]
Since $\varphi_x(s, u^{\prime }_i)\in \mathcal{R}_{\leq t, \Omega'}(x)$ we have that $y\in N_{\varepsilon}(\mathcal{R}_{\leq t,
\Omega^{\prime }}(x))$ and consequently that 
$$\mathcal{R}_{\leq t, \Omega}(x)\subset N_{\varepsilon}(\mathcal{R}_{\leq t, \Omega^{\prime }}(x))$$
as stated.

\textbf{Claim 3:} For any $\varepsilon>0$ there exists $\delta_2>0$ such
that $N_{\delta_2}(\Omega)\subset \mathrm{int}\widehat{\Omega}$ and 
\[
\Omega^{\prime }\subset N_{\delta_2}(\Omega)\Rightarrow \mathcal{R}_{\leq t,
\Omega^{\prime }}(x)\subset N_{\varepsilon}(\mathcal{R}_{\leq t, \Omega}(x)).
\]

Since $\mathcal{R}_{\leq t, \Omega^{\prime }}(x)=\mathrm{cl}\left(\mathcal{R}%
_{\leq t, \Omega^{\prime }}^{\mathrm{PC}}(x)\right)$ it is enough for us to
show that for any $\varepsilon>0$ there is $\delta_2>0$ such $%
N_{\delta_2}(\Omega)\subset \mathrm{int}\widehat{\Omega}$ and 
\[
\Omega^{\prime }\subset N_{\delta_2}(\Omega)\Rightarrow \mathcal{R}_{\leq t,
\Omega^{\prime }}^{\mathrm{PC}}(x)\subset N_{\varepsilon}\left(\mathcal{R}%
_{\leq t, \Omega}(x)\right)
\]

Take $\delta_2>0$ be such that $\delta_2<\epsilon/2M$ and $%
N_{\delta_2}(\Omega)\subset \mathrm{int}\,\widehat{\Omega}$. For any $%
\Omega^{\prime }\subset N_{\delta_2}(\Omega)$ let $z\in \mathcal{R}^{\mathrm{PC}}_{\leq t,
\Omega^{\prime }}(x)$ and $s\in[0, t]$, $u^{\prime }\in \mathcal{U}%
_{\Omega^{\prime }}^{\mathrm{PC}}$ such that $z=\varphi_x(s, u^{\prime })$.
If $c^{\prime }_1, \ldots, c_m^{\prime }$ are the values assumed by $u^{\prime }$
in $[-T, T]$, there are $c_1, \ldots, c_m\in \Omega$ such that $%
|c_i-c_i^{\prime }|<\delta_2$. Define the piecewise constant function $u\in%
\mathcal{U}_{\Omega}$ by 
\[
u(s):=\left \{%
\begin{array}{cc}
c_i, & \mbox{ if } u^{\prime }(s)=c^{\prime }_i \; \; \mbox{ and }\; \; s\in[%
0, t] \\ 
p, & \mbox{ if } s\in \mathbb{R}\setminus [-T, T],%
\end{array}%
\right.
\]
where $p\in \Omega$ is an arbitrary point. Therefore for $1\leq j\leq k_i$, $%
1\leq i\leq m$ we have 
\[
\left|\int_{\mathbb{R}}\langle u^{\prime }(s)-u(s), x_{i, j}(s)\rangle 
\mathrm{ds}\right|
\]
\[
\delta_2\int_{-T}^T\left|x_{i, j}(s)\right|\mathrm{ds}+\mathrm{diam}\widehat{%
\Omega}\int_{\mathbb{R}\setminus[-T, T]}\left|x_{i, j}(s)\right| \mathrm{ds}%
<\epsilon
\]
showing that $u^{\prime }\in W_{u, \epsilon}$ and implying that $u^{\prime
}\in W_{u_i, \gamma_i}(x_{i, 1}, \ldots, x_{i, k_i})$ for some $1\leq i\leq
m $.

Therefore, 
\[
\varrho(\varphi_x(s, u^{\prime }), \varphi_x(s, u_i))<\varepsilon
\]
showing that $z\in N_{\varepsilon}(\mathcal{R}_{\leq t, \Omega})$. Since $%
z\in \mathcal{R}_{\leq t, \Omega^{\prime }}^{\mathrm{PC}}(x)$ was arbitrary
we conclude that 
\[
\mathcal{R}_{t, \Omega^{\prime }}^{\mathrm{PC}}(x)\subset N_{\varepsilon}(%
\mathcal{R}_{t, \Omega}(x))
\]
as claimed.

\bigskip

\textbf{Claim 4:} The map $\Omega \mapsto \mathcal{R}_{\leq t, \Omega}(x)$
is continuous in the Hausdorff measure.
\bigskip

For given $\Omega$ and $\varepsilon>0$ let $\delta=\min \{ \delta_1,
\delta_2/2\}$ where $\delta_1, \delta_2$ are given in the claims 1. and 2.
We have that 
\[
d_H(\Omega, \Omega^{\prime })<\delta \; \; \; \Leftrightarrow \;\;\;\Omega \subset
N_{\delta}(\Omega^{\prime }) \mbox{ and }\Omega^{\prime }\subset
N_{\delta}(\Omega).
\]
By Claim 3. we get that 
\[
\Omega^{\prime }\subset N_{\delta}(\Omega) \Rightarrow \mathcal{R}_{\leq t,
\Omega^{\prime }}(x)\subset N_{\varepsilon}(\mathcal{R}_{\leq t, \Omega}(x)).
\]
Moreover, 
\[
\Omega^{\prime }\subset N_{\delta}(\Omega)\Rightarrow
N_{\delta}(\Omega^{\prime })\subset N_{2\delta}(\Omega)\subset
N_{\delta_2}(\Omega)\subset \mathrm{int}\widehat{\Omega}
\]
and so, Claim 2 implies that 
\[
\Omega \subset N_{\delta}(\Omega^{\prime })\Rightarrow \mathcal{R}_{\leq t,
\Omega}(x)\subset N_{\varepsilon}(\mathcal{R}_{\leq t, \Omega^{\prime }}(x))
\]
showing 
\[
d_H(\Omega, \Omega^{\prime })<\delta \Rightarrow \varrho_H(\mathcal{R}_{\leq
t, \Omega}(x), \mathcal{R}_{\leq t, \Omega^{\prime }}(x))<\varepsilon
\]
as stated.
\end{proof}

\begin{lemma}
The map $t\mapsto \mathcal{R}_{\leq t, \Omega}(x)$ is continuous.
\end{lemma}

\begin{proof}
For any $u\in \mathcal{U}_{\Omega}$ there exists by continuity $\delta_u>0$
and $V_u$ a neighborhood of $u$ in $\mathcal{U}_{\Omega}$ such that 
\[
|s-t|<\gamma_u\; \mbox{ and }\;v\in V_u\Rightarrow \varrho(\varphi_x(t, u),
\varphi_x(s, v))<\varepsilon/2.
\]
Since $\mathcal{U}_{\Omega}$ is compact, there exist $V_1, \ldots, V_n$ such
that 
\[
\mathcal{U}_{\Omega}=\bigcup_{i=1}^n V_n.
\]
By taking $\gamma=\min_{1\leq i\leq n}\{ \gamma_{u_i}\}$ we have that for
any $s, s^{\prime }\in (t-\gamma/2, t+\gamma/2)$ and $u, u^{\prime }\in V_i$
that 
\begin{equation}  \label{metric}
\varrho(\varphi_x(s, u), \varphi_x(s^{\prime }, u^{\prime }))<\varepsilon,
\; \; \mbox{ for some }\; \;1\leq i\leq n.
\end{equation}

By taking $\delta=\gamma/2$ we have for $s\in(t-\delta, t+\delta)$ that

\begin{itemize}
\item[(i)] if $s\geq t$ then 
\[
\mathcal{R}_{\leq t, \Omega}(x)\subset \mathcal{R}_{\leq s,
\Omega}(x)\subset N_{\varepsilon}\left(\mathcal{R}_{\leq s,
\Omega}(x)\right).
\]
Also, if $z\in \mathcal{R}_{\leq s, \Omega}(x)$ we have that $%
z=\varphi(s^{\prime }, x, u)$ for some $s^{\prime }\in[0, s]$ and $u\in%
\mathcal{U}_{\Omega}$. Then

\begin{itemize}
\item[(a)] if $s^{\prime }\leq t$ we have $z\in \mathcal{R}_{\leq t,
\Omega}(x)\subset N_{\varepsilon}\left(\mathcal{R}_{\leq t,
\Omega}(x)\right) $;

\item[(b)] if $s^{\prime }> t$ we have that $s^{\prime }\in [t, s]\subset
(t-\delta, t+\delta)$ and so, by taking $V_i$ such that $u\in V_i$ equation (%
\ref{metric}) gives that 
\[
\varrho(\varphi_x(s^{\prime }, u), \varphi_x(t, u^{\prime }))<\varepsilon,
\; \; \; \mbox{ for any }\; \;u^{\prime }\in V_i
\]
implying that $z\in N_{\varepsilon}\left(\mathcal{R}_{\leq t,
\Omega}(x)\right)$.
\end{itemize}

Since $z\in \mathcal{R}_{\leq s, \Omega}(x)$ was arbitrary we conclude that 
\[
\mathcal{R}_{\leq s, \Omega}(x)\subset N_{\varepsilon}(\mathcal{R}_{\leq t,
\Omega}(x));
\]

\item[(ii)] if $s<t$ we have that 
\[
\mathcal{R}_{\leq s, \Omega}(x)\subset \mathcal{R}_{\leq t,
\Omega}(x)\subset N_{\varepsilon}\left(\mathcal{R}_{\leq t,
\Omega}(x)\right).
\]
Also, for any $z\in \mathcal{R}_{\leq t, \Omega}(x)$ we have that $%
z=\varphi(t^{\prime }, x, u)$ for some $t^{\prime }\in[0, t]$ and $u\in%
\mathcal{U}_{\Omega}$. Then

\begin{itemize}
\item[(c)] if $t^{\prime }\leq s$ we have $z\in \mathcal{R}_{\leq s,
\Omega}(x)\subset N_{\varepsilon}\left(\mathcal{R}_{\leq s,
\Omega}(x)\right) $;

\item[(d)] if $t^{\prime }> s$ we have that $t^{\prime }\in [s, t]\subset
(t-\delta, t+\delta)$ and so, by taking $V_i$ such that $u\in V_i$ equation (%
\ref{metric}) gives that 
\[
\varrho(\varphi_x(t^{\prime }, u), \varphi_x(t, u^{\prime }))<\varepsilon,
\; \; \; \mbox{ for any }\; \;u^{\prime }\in V_i
\]
implying that $z\in N_{\varepsilon}\left(\mathcal{R}_{\leq s,
\Omega}(x)\right)$.
\end{itemize}

Since $z\in \mathcal{R}_{\leq t, \Omega}(x)$ was arbitrary we conclude that 
\[
\mathcal{R}_{\leq t, \Omega}(x)\subset N_{\varepsilon}(\mathcal{R}_{\leq s,
\Omega}(x));
\]
\end{itemize}

Therefore, if $s\in(t-\delta, t+\delta)$ we have that 
\[
\mathcal{R}_{\leq t, \Omega}(x)\subset N_{\varepsilon}(\mathcal{R}_{\leq s,
\Omega}(x)) \; \; \; \mbox{ and }\; \; \; \mathcal{R}_{\leq s,
\Omega}(x)\subset N_{\varepsilon}(\mathcal{R}_{\leq t, \Omega}(x))
\]
if and only if 
\[
\varrho_H\left(\mathcal{R}_{\leq t, \Omega}(x), \mathcal{R}_{\leq s,
\Omega}(x)\right)<\varepsilon
\]
showing the result.
\end{proof}

\begin{lemma}
The map $x\mapsto \mathcal{R}_{\leq t, \Omega}(x)$ is continuous.
\end{lemma}

\begin{proof}
Let $x\in M$ and $t>0$ fixed and consider $\varepsilon>0$. By continuity of
the solutions and compacity of $[0, t]\times \mathcal{U}_{\Omega}$ we can
find $\delta>0$ such that 
\[
y\in B(x, \delta)\Rightarrow \varrho(\varphi_{s, u}(x), \varphi_{s,
u}(y))<\varepsilon, \; \; \mbox{ for all }(s, u)\in [0, t]\times \mathcal{U}%
_{\Omega}\; \;
\]
where $\varphi_{t, u}(x):=\varphi(t, x, u)$. Then, for $z\in \mathcal{R}%
_{\leq t, \Omega}(x)$ let $s\in [0, t]$ and $u\in \mathcal{U}_{\Omega}$ such
that $z=\varphi_{s, u}(x)$. If $y\in B(x, \delta)$ we have by above that $%
\varrho(\varphi_{s, u}(x), \varphi_{s, u}(y))<\varepsilon$ showing that $%
z\in \mathcal{R}_{\leq t, \Omega}(y)$ and implying that 
\[
\mathcal{R}_{\leq t, \Omega}(x)\subset N_{\varepsilon}\left(\mathcal{R}%
_{\leq t, \Omega}(y)\right).
\]
In an analogous way we can show that 
\[
\mathcal{R}_{\leq t, \Omega}(y)\subset N_{\varepsilon}\left(\mathcal{R}%
_{\leq t, \Omega}(x)\right)
\]
and so 
\[
\varrho(x, y)<\delta \Rightarrow \varrho_H\left(\mathcal{R}_{\leq t,
\Omega}(x), \mathcal{R}_{\leq t, \Omega}(y)\right)
\]
concluding the proof.
\end{proof}

\bigskip
Now we are able to prove our main result.

\begin{theorem}
The map $(t, x, \Omega)\mapsto \mathcal{R}_{\leq t, \Omega}(x)$ is a
continuous map.
\end{theorem}

\begin{proof}
Since 
\[
\varrho _{H}\left( \mathcal{R}_{\leq t^{\prime },\Omega ^{\prime
}}(x^{\prime }),\mathcal{R}_{t,\Omega }(x)\right) \leq \varrho _{H}\left( 
\mathcal{R}_{\leq t^{\prime },\Omega ^{\prime }}(x^{\prime }),\mathcal{R}%
_{\leq t^{\prime },\Omega ^{\prime }}(x)\right) 
\]%
\[
+\varrho _{H}\left( \mathcal{R}_{\leq t^{\prime },\Omega ^{\prime }}(x),%
\mathcal{R}_{\leq t^{\prime },\Omega }(x))\right) +\varrho _{H}\left( 
\mathcal{R}_{\leq t^{\prime },\Omega }(x),\mathcal{R}_{\leq t,\Omega
}(x)\right) 
\]%
the result follows from the above lemmas.
\end{proof}

\end{document}